\documentclass{amsart}
\usepackage{amssymb,amsmath,latexsym}
\usepackage[dvips]{graphicx}
\usepackage{times}

\newtheorem{thm}{Theorem}[section]

\newcommand{\cl}{\hbox{\rm cl}}
\newcommand{\join}{\lor}
\newcommand{\meet}{\land}

\newcommand{\mcZ}{\mathcal{Z}}

\title{A Note on the Sticky Matroid Conjecture}
                   
\date{\today} \author[J.\ Bonin]{Joseph E.~Bonin} \address
{Department of Mathematics\\ The George Washington University\\
  Washington, D.C. 20052} \email
{jbonin@gwu.edu} 

\begin{document}

\begin{abstract}
  A matroid is sticky if any two of its extensions by disjoint sets
  can be glued together along the common restriction (that is, they
  have an amalgam).  The sticky matroid conjecture asserts that a
  matroid is sticky if and only if it is modular.  Poljak and Turzik
  proved that no rank-$3$ matroid having two disjoint lines is sticky.
  We show that, for $r\geq 3$, no rank-$r$ matroid having two disjoint
  hyperplanes is sticky.  These and earlier results show that the
  sticky matroid conjecture for finite matroids would follow from a
  positive resolution of the rank-$4$ case of a conjecture of Kantor.
\end{abstract}

\maketitle

\section{Introduction}

A matroid $M$ is \emph{sticky} if whenever the restrictions of any two
matroids $N$ and $N'$ to $E(N)\cap E(N')$ are equal to each other and
isomorphic to $M$, then $N$ and $N'$ have an amalgam, that is, a
matroid on the set $E(N)\cup E(N')$ having both $N$ and $N'$ as
restrictions.  Modular matroids are sticky; see~\cite[Theorem
12.4.10]{oxley}.  The sticky matroid conjecture, posed in~\cite{pt},
asserts the converse: sticky matroids are modular.

Poljak and Turzik~\cite{pt} showed that the conjecture holds for
rank-$3$ matroids.  Bachem and Kern~\cite{bk} showed that a rank-$4$
matroid is not sticky if the intersection of some pair of planes is a
point.  We prove that, for $r\geq 3$, a rank-$r$ matroid is not sticky
if it has a pair of disjoint hyperplanes.

Lemma 6 in~\cite{bk} says the conjecture holds for all matroids having
the following property.
\begin{quote}
  \emph{The intersection property}: whenever $(X,Y)$ is a non-modular
  pair of flats of $M$, there is a modular cut of $M$ that includes
  $X$ and $Y$ but not $X\cap Y$.
\end{quote}
We give a counterexample to an assertion used in the proof of the
lemma; we also show that the lemma is correct.  Using this lemma,
Bachem and Kern showed that the sticky matroid conjecture is true if
and only if it holds for rank-$4$ matroids.  They also show that for
rank-$4$ matroids, the intersection property is equivalent to the
following condition.
\begin{quote}
  \emph{The bundle condition}: given four lines in rank $4$ with no
  three coplanar, if five of the six pairs of lines are coplanar, then
  so is the sixth pair.
\end{quote}

Thus, future work on the conjecture can focus on rank-$4$ matroids in
which each pair of planes intersects in a line and in which the bundle
condition fails.  Modular matroids and their restrictions satisfy the
bundle condition, so these results imply that the sticky matroid
conjecture for finite matroids would follow from a positive resolution
of the rank-$4$ case of Kantor's conjecture~\cite{kantor}: for
sufficiently large $r$, if a finite rank-$r$ matroid $M$ has the
property that each pair of hyperplanes intersects in a flat of rank
$r-2$, then $M$ has an extension to a modular matroid.
(See~\cite[Example 5]{kantor} for the necessity of the finiteness
hypothesis in Kantor's conjecture.)

The results and proofs below apply to both finite and infinite
matroids.

\section{Background}

We assume familiarity with basic matroid theory, including
single-element extensions and modular cuts~\cite{crapo,oxley}.  We
will use the formulation of matroids via cyclic flats and their ranks
stated below.  A \emph{cyclic set} of a matroid is a union of
circuits.  It is easy to see that the cyclic flats of a matroid $M$
form a lattice; we denote this lattice by $\mcZ(M)$.
Brylawski~\cite{affine} observed that a matroid is determined by its
cyclic flats and their rank; the following result
from~\cite{juliethesis,cyclic} carries this further.

\begin{thm}\label{thm:axioms}
  Let $\mcZ$ be a collection of subsets of a set $S$ and let $r$ be
  an integer-valued function on $\mcZ$.  There is a matroid for
  which $\mcZ$ is the collection of cyclic flats and $r$ is the rank
  function restricted to the sets in $\mcZ$ if and only if
\begin{itemize}
\item[(Z0)] $\mcZ$ is a lattice under inclusion,
\item[(Z1)] $r(0_{\mcZ})=0$, where $0_{\mcZ}$ is the least element of
  $\mcZ$,
\item[(Z2)] $0<r(Y)-r(X)<|Y-X|$ for all sets $X,Y$ in $\mcZ$ with
  $X\subsetneq Y$, and  
\item[(Z3)] for all pairs of incomparable sets $X,Y$ in $\mcZ$,
  \begin{equation}\label{eq:semimod}  
    r(X)+r(Y)\geq r(X\join Y) + r(X\meet Y) + |(X\cap Y) - (X\meet
    Y)|.
  \end{equation}
\end{itemize}
\end{thm}

The V\'amos matroid (Figure~\ref{vamos}) motivates our constructions.
This rank-$4$ matroid on the set $\{a,a',b,b',c,c',d,d'\}$ has as its
nonempty, proper cyclic flats, all of rank $3$, all sets of the form
$\{x,x',y,y'\}$ except $\{a,a',d,d'\}$.  It does not satisfy the
bundle condition.

\begin{figure}
\begin{center}
\includegraphics[width = 2.25truein]{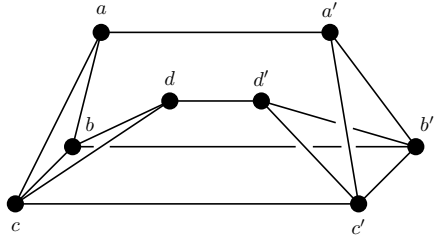}
\end{center}
\caption{The V\'amos matroid.}\label{vamos}
\end{figure}

\section{Results}

Bachem and Kern~\cite{bk} showed that contractions of sticky matroids
are sticky.  They noted a corollary of this result and that of Poljak
and Turzik: if two planes in a rank-$4$ matroid intersect in a point,
then the matroid is not sticky.  The case $r=4$ of the following
result addresses disjoint planes; the case $r=3$ is the result of
Poljak and Turzik.

\begin{thm}\label{thm:planes}
  For $r\geq 3$, a rank-$r$ matroid having two disjoint hyperplanes is
  not sticky.
\end{thm}

\begin{proof}
  Let $H$ and $H'$ be disjoint hyperplanes in a matroid $M$ of rank
  $r$.  In $M$, the set $\mathcal{M}=\{H,H',E(M)\}$ is a modular cut.
  If $r>3$, then, in the extension to $E(M)\cup p$ corresponding to
  $\mathcal{M}$, the set $\{H\cup p,H'\cup p,E(M)\cup p\}$ is a
  modular cut.  Continuing this way yields an extension $M_P$ of $M$
  to $E(M)\cup P$ in which $P$ is an independent set of size $r-2$
  with $P\subseteq \cl_{M_P}(H)\cap \cl_{M_P}(H')$.  To show that $M$
  is not sticky, we construct an extension $N$ of $M$ that contains no
  elements of $P$ and so that $N$ and $M_P$ have no amalgam.

  Add a point freely to $H$ (respectively, $H'$) if it is not already
  cyclic.  This gives a matroid $M'$ in which the flats
  $H_1=\cl_{M'}(H)$, $H_2=\cl_{M'}(H')$, and $E(M')$ are cyclic.
  (Constructing $M'$ is not essential; it makes the proof slightly
  easier to state.)  Fix two $(r-1)$-element sets $A$ and $B$ that are
  disjoint from each other and from $E(M')$.  We define the extension
  $N$ of $M'$ by its lattice of cyclic flats and their ranks.  The
  cyclic flats of $N$ are those of $M'$ (these have the same ranks in
  the two matroids) along with
  \begin{enumerate}
  \item $E(M')\cup A \cup B$ of rank $r+1$, and
  \item $H_1\cup A$, $H_1\cup B$, $H_2\cup A$, and $H_2\cup B$, all of
    rank $r$.
  \end{enumerate}
  (See Figure~\ref{rank4b}.)  To show that the resulting collection
  $\mcZ(N)$ is a lattice, it suffices to show that each pair $X,Y\in
  \mcZ(N)$ of incomparable sets has a join; if both $X$ and $Y$ are in
  $\mcZ(M')$, then their join is as in the lattice $\mcZ(M')$,
  otherwise it is $E(M')\cup A \cup B$.  Properties (Z1) and (Z2) in
  Theorem~\ref{thm:axioms} are easy to see, so we turn to (Z3).  Since
  $\mcZ(M')$ is a sublattice of $\mcZ(N)$ and since the function $r$
  on $\mcZ(N)$ extends that on $\mcZ(M')$,
  inequality~(\ref{eq:semimod}) in property (Z3) holds if
  $X,Y\in\mcZ(M')$.  Inequality~(\ref{eq:semimod}) is easy to check
  when $X$ and $Y$ are sets in item (2) above.  Lastly, by symmetry it
  suffices to consider $X=H_1\cup A$ and an incomparable flat $Y \in
  \mcZ(M')$.  Inequality~(\ref{eq:semimod}) follows easily in this
  case from two observations: (i) the flat $(H_1\cup A)\cap Y =
  H_1\cap Y$ of $M'$ has rank at most $r(Y)-1$ and (ii) $r(H_1\cap Y)
  = r(H_1\meet Y) + |(H_1\cap Y) - (H_1\meet Y)|$.  Thus, property
  (Z3) holds, so $N$ is indeed a matroid.

  \begin{figure}
    \begin{center}
      \includegraphics[width = 3.0truein]{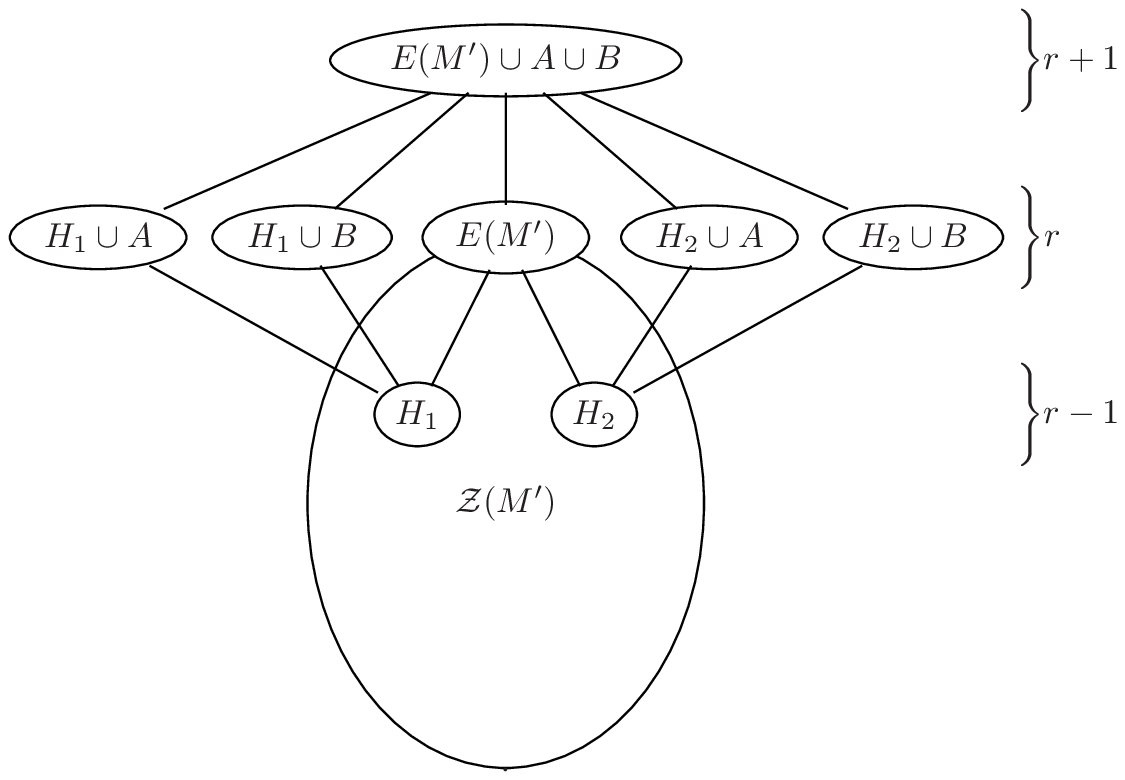}
    \end{center}
    \caption{The lattice $\mcZ(N)$ in the proof of
      Theorem~\ref{thm:planes}.}\label{rank4b}
  \end{figure}

  Finally, we prove that $N$ and $M_P$ have no amalgam by showing that
  in any extension $N'$ of $N$ to $E(N)\cup P$ with $P\subseteq
  \cl_{N'}(H)\cap \cl_{N'}(H')$ (i.e., $\cl_{N'}(H_1)\cap
  \cl_{N'}(H_2)$), we have $r_{N'}(P)\leq r-3$, which conflicts with
  $r_{M_P}(P)=r-2$.  Since $P\subseteq \cl_{N'}(H_1\cup A)$ and
  $P\subseteq \cl_{N'}(H_2\cup A)$, and since $(H_1\cup A,H_2\cup A)$
  is a modular pair of flats in $N$, we get $P\subseteq \cl_{N'}(A)$.
  Similarly, $P\subseteq \cl_{N'}(B)$.  Semimodularity gives
  $$r_{N'}(A\cup P)+ r_{N'}(B\cup P) 
  \geq r_{N'}(A\cup B\cup P) + r_{N'}(P),$$ that is $2(r-1)\geq r+1 +
  r_{N'}(P)$, so, as claimed, $r_{N'}(P)\leq r-3$.
\end{proof}

We now turn to~\cite[Lemma 6]{bk} and the flawed assertion used in its
proof.  Recast in matroid terms, the assertion is the following.
\begin{quote}
  If a rank-$r$ matroid $M$ contains three rank-$(r-2)$ flats $D_1$,
  $D_2$, and $D_3$, and a line $\ell_4$ such that $D_1\cup D_2$ spans
  $M$ but $D_1\cup D_3$, $D_2\cup D_3$, and $D_i\cup\ell_4$, for
  $i\in\{1,2,3\}$, span five different hyperplanes, then $M$ does not
  have the intersection property.  \cite[Example (b), p.~14.]{bk}
\end{quote}
For a counterexample, consider the rank-$5$ matroid $M$ that is
represented by the following matrix over $\mathbb{R}$ (or over any
field of characteristic other than $2$ or $3$).
\begin{equation*}
  \left(\begin{matrix}
      0 & 0 & 1 \\
      1 & 1 & 1 \\
      2 & 3 & 4 \\
      0 & 0 & 0 \\
      0 & 0 & 0
    \end{matrix}\right.
  \left|\, \begin{matrix}
      0 & 0 & 1 \\
      0 & 0 & 0 \\
      0 & 0 & 0 \\
      1 & 1 & 1 \\
      2 & 3 & 4 
    \end{matrix}\right.
  \left|\, \begin{matrix}
      0 & 0 & 1 \\
      0 & 0 & 0 \\
      1 & 1 & 1 \\
      2 & 3 & 4 \\
      0 & 0 & 0 
    \end{matrix}
  \right.
  \left|\, \begin{matrix}
      1 & 0 \\
      1 & 1 \\
      1 & 1 \\
      1 & 1 \\
      1 & 1  
    \end{matrix}
\right)
\end{equation*}
The bars separate four groups of columns corresponding to the three
planes $D_1$, $D_2$, $D_3$, and the line $\ell_4$ respectively.  Since
$M$ is representable over $\mathbb{R}$, it has the intersection
property.  The following observations show that it satisfies the
hypotheses of the claim.  The union $D_1\cup\ell_4$ is the flat of
rank $4$ that consists of all columns in which the last two entries
are equal.  Similarly, $D_2\cup\ell_4$ is the rank-$4$ flat of that
consists of all columns in which the second and third entries are
equal, and $D_3\cup\ell_4$ is the rank-$4$ flat that consists of all
columns in which the second and last entries are equal.  The union
$D_1\cup D_2$ spans $M$, while $D_1\cup D_3$ is the hyperplane
consisting of all columns whose last entry is zero, and $D_2\cup D_3$
is the hyperplane consisting of all columns whose second entry is
zero.

We next offer a proof of~\cite[Lemma 6]{bk}.

\begin{thm}\label{thm:ip}
  For $r\geq 4$, if a rank-$r$ matroid $M$ has a line $\ell$ and
  hyperplane $H$ that are disjoint, then $M$ has a loopless extension
  $N$ with $\cl_{N'}(\ell)\cap \cl_{N'}(H)=\emptyset$ for all loopless
  extensions $N'$ of $N$.  Thus, if $M$ also has the intersection
  property, then it is not sticky.
\end{thm}

\begin{proof}
  Let $A$ be an $(r-3)$-element set disjoint from $E(M)$.  Obtain $M'$
  from $M$ by adding the elements of $A$ freely to $H$.  Let $H' =
  H\cup A$.  Fix $(r-1)$-element supersets $D_1$ and $D_2$ of $A$ with
  $D_1-A$ and $D_2-A$ disjoint from each other and from $E(M')$.  The
  ground set of $N$ will be $E(M') \cup D_1 \cup D_2$.  We obtain
  $\mcZ(N)$ by adjoining to $\mcZ(M')$ the sets $E(M') \cup D_1 \cup
  D_2$ (of rank $r+1$) and $D_1\cup H'$, $D_1\cup \ell$, $D_2\cup H'$,
  and $D_2\cup \ell$ (all of rank $r$).  As above, properties
  (Z0)--(Z3) of Theorem~\ref{thm:axioms} hold.

  We now show that if $N'$ is a single-element extension of $N$ on the
  set $E(N)\cup \{q\}$ and if $q\in\cl_{N'}(\ell)\cap \cl_{N'}(H')$,
  then $q$ is a loop of $N'$.  Note that $(D_1\cup \ell,D_1\cup H')$
  is a modular pair of flats in $N$ and $q$ is in the closures, in
  $N'$, of both sets; therefore $q\in \cl_{N'}(D_1)$.  Similarly,
  $q\in \cl_{N'}(D_2)$.  Since $(D_1,D_2)$ is a modular pair of flats
  in $N$, we get $q\in\cl_{N'}(A)$.  The elements of $A$ were added
  freely to $H$, so $(A,\ell)$ is a modular pair of flats of $N$.
  Moreover, $A$ and $\ell$ are disjoint and $q\in\cl_{N'}(\ell)\cap
  \cl_{N'}(A)$, so it follows that $a$ is a loop of $N'$.
\end{proof}

Bachem and Kern~\cite{bk} showed that a rank-$4$ matroid satisfies the
intersection property if and only if it satisfies the bundle
condition.  (A careful reading of their proof reveals gaps; however,
the gaps can be filled with the type of argument they use.)  One
direction of this equivalence is transparent.  To highlight how the
bundle condition enters from the perspective of modular cuts, we give
a brief alternate proof of the more substantial direction.

\begin{thm}\label{thm:bundleimpliesip}
  For rank-$4$ matroids, the bundle condition implies the intersection
  property.
\end{thm}

\begin{proof}
  Let $M$ be a rank-$4$ matroid in which the bundle condition holds.
  We need to show that for each non-modular pair of flats $(X,Y)$ in
  $M$, there is a modular cut of $M$ that contains $X$ and $Y$ but not
  $X\cap Y$.  If $X$ and $Y$ are planes, then $\{X,Y,E(M)\}$ is the
  required modular cut.  If $X$ is a plane, $Y$ is a line, and $Y$ is
  not coplanar with any line in $X$, then the filter of flats
  generated by $X$ and $Y$ is the required modular cut. Thus, only the
  case of disjoint coplanar lines remains to be addressed.

  Let $\ell_1$ and $\ell_2$ be disjoint lines in the plane $P$ of $M$.
  Consider the set $\mathcal{L}$ that is the union of the following
  three sets: $\{\ell_1,\ell_2\}$, the set $\mathcal{L}_{\bar{P}}$ of
  all lines not in the plane $P$ that are coplanar with both $\ell_1$
  and $\ell_2$, and the set $\mathcal{L}_{P}$ of all lines in $P$ that
  are coplanar with at least one line in $\mathcal{L}_{\bar{P}}$.  The
  bundle condition shows that $\mathcal{L}$ has the following
  properties.
  \begin{itemize}
  \item[(a)] All lines in  $\mathcal{L}_{\bar{P}}$  are coplanar. 
  \item[(b)] Lines in $\mathcal{L}_{P}$ are coplanar with all lines in
    $\mathcal{L}_{\bar{P}}$.
  \item[(c)] Any line that is in two distinct planes with two lines of
    $\mathcal{L}$ is also in $\mathcal{L}$.
  \end{itemize}
  Furthermore, any two lines in $\mathcal{L}$ are disjoint.  It
  follows that the filter that $\mathcal{L}$ generates is a modular
  cut.  Thus, the intersection property holds.
\end{proof}

\end{document}